\documentclass[12pt]{amsart}

\usepackage{amssymb}
\usepackage{amsmath}
\usepackage{amsthm}
\usepackage{amsfonts}
\usepackage{amscd}

\usepackage{enumerate}

\usepackage{graphicx}

\makeatletter
\@namedef{subjclassname@2010}{%
  \textup{2010} Mathematics Subject Classification}
\makeatother

\theoremstyle{plain} 
\newtheorem{theorem}{Theorem}[section]
\newtheorem{proposition}[theorem]{Proposition}
\newtheorem{lemma}[theorem]{Lemma}
\newtheorem{corollary}[theorem]{Corollary}

\theoremstyle{remark}

\theoremstyle{definition}
\newtheorem{definition}[theorem]{Definition}

\numberwithin{equation}{section}

\DeclareMathOperator{\Stab}{Stab}
\DeclareMathOperator{\Spec}{Spec}
\DeclareMathOperator{\Orb}{Orb}

\DeclareMathOperator{\Gal}{Gal}
\DeclareMathOperator{\tr}{tr}
\DeclareMathOperator{\Prob}{Prob}
\DeclareMathOperator{\Per}{Per}
\DeclareMathOperator{\Frob}{Frob}

\DeclareMathOperator{\FPP}{FPP}
\DeclareMathOperator{\Fix}{Fix}
\DeclareMathOperator{\Norm}{N}

\newcommand{\fp}{ {\mathfrak p} }
\newcommand{\fm}{ {\mathfrak m} }
\newcommand{\fq}{ {\mathfrak q} }

\newcommand{\fc}{ {\mathfrak c} }

\newcommand{\fo}{ {\mathfrak o} }

\newcommand{\cS}{ {\mathcal S} }

\newcommand{\cU}{ {\mathcal U} }

\newcommand{\F} { {\mathbb F}} 
\newcommand{\bP} { {\mathbb P}}

\newcommand{\tp}{ {\tilde \varphi}}

\begin{document}

%%%%% To ease editing, for IMPAN journals add:

\baselineskip=17pt

%%%%%%%%%%%

%% In the running head, replace first names by initials 
%% and give an abbreviation of the title.

\title[Fixed point proportions for non-geometric iterated extensions]{Fixed point proportions for Galois groups of non-geometric iterated extensions}

\author[J. Juul]{Jamie Juul}
\address{
	Department of Mathematics and Statistics\\
	Amherst College\\
	Amherst, MA 01002\\
	USA
}
\email{jamie.l.rahr@gmail.com}

\date{}

\begin{abstract}
Given a map $\varphi:\mathbb{P}^1\rightarrow \mathbb{P}^1$ of degree greater than 1 defined over a number field $k$, one can define a map $\varphi_\mathfrak{p}:\mathbb{P}^1(\mathfrak{o}_k/\mathfrak{p})\rightarrow \mathbb{P}^1(\mathfrak{o}_k/\mathfrak{p})$ for each prime $\mathfrak{p}$ of good reduction, induced by reduction modulo $\mathfrak{p}$. It has been shown that for a typical $\varphi$ the proportion of periodic points of $\varphi_\mathfrak{p}$ should tend to $0$ as $|\mathbb{P}^1(\mathfrak{o}_k/\mathfrak{p})|$ grows. In this paper, we extend previous results to include a weaker set of sufficient conditions under which this property holds. We are also able to show that these conditions are necessary for certain families of functions, for example, functions of the form $\varphi(x)=x^d+c$, where $0$ is not a preperiodic point of this map. We study the proportion of periodic points by looking at the fixed point proportion of the Galois groups of certain extensions associated to iterates of the map.

\end{abstract}



\maketitle

\section{Introduction}

Let $k$ be a field and $\varphi(x)$ a rational map of degree $d$ greater than 1 with coefficients in $k$. Consider $\varphi^n(x)-t$, where $\varphi^n(x)$ is the $n$-th iterate of $\varphi(x)$ and $t$ is transcendental over $k$. If $\varphi^n(x)-t$ is separable for each $n$, adjoining the roots of $\varphi^n(x)-t$ to $k(t)$ gives us a Galois extension of $k(t)$. 

In order to understand certain dynamic and number theoretic properties of $\varphi$ it is useful to understand the structure of the Galois groups of these extensions (see for example, \cite{odoni}\cite{Jones2}\cite{HJM}\cite{Jones}\cite{JKMT}\cite{JO}). The proportion of elements of the Galois group fixing some root of $\varphi^n(x)-t$, called the \textit{fixed point proportion},  is particularly useful in applications. This paper generalizes some of the results from \cite{JKMT}, which examines proportions of periodic points by studying the fixed point proportion of these Galois groups.

If $\varphi(x)$ has coefficients in a number field $k$, then for any prime $\fp$ of $\fo_k$ (of good reduction) we can consider the map $\varphi_\fp:\bP^1(\fo_k/\fp)\rightarrow \bP^1(\fo_k/\fp)$ induced by reduction modulo $\fp$. Since $\bP^1(\fo_k/\fp)$ is finite, the orbit of any point in $\bP^1(\fo_k/\fp)$ will be finite and hence must contain a cycle by the pigeonhole principle. Thus, any point in $\bP^1(\fo_k/\fp)$ is \textit{preperiodic} under this map. It is natural to ask: what proportion of elements of $\bP^1(\fo_k/\fp)$ are \textit{periodic} under this map, that is, for what proportion of $\alpha \in \bP^1(\fo_k/\fp)$ does there exist $n>1$ such that $\varphi^n(\alpha)=\alpha$?

%If we assume that as we range over various choices of $\fp$, $\varphi_\fp$ behaves randomly, then we expect that for $\alpha\in\bP^1(\fo_k/\fp)$, the
%size of the orbit $\cO_\varphi(\alpha) = \{ \alpha, \varphi_\fp(\alpha), \dots,
%\varphi^n_\fp(\alpha), \dots \}$ will be about $\sqrt{|\bP^1(\fo_k/\fp)|}$.  This heuristic is based on the birthday problem; if we choose items from a set of size $|\bP^1(\fo_k/\fp)|$ at random we expect to get a repetition after about $\sqrt{|\bP^1(\fo_k/\fp)|}$ choices (see \cite{FO}\cite{Bach}\cite{Silver}\cite{BGH}).  Hence, one might guess that there is about a
%$1/\sqrt{|\bP^1(\fo_k/\fp)|}$ chance that a given $\alpha$ is $\varphi_\fp$-periodic, and that
%the proportion of $\varphi_\fp$-periodic points in $\bP^1(\fo_k/\fp)$ is about $1/\sqrt{|\bP^1(\fo_k/\fp)|}$.

Heuristics suggest that for a typical $\varphi$, we should expect the proportion of $\varphi_\fp$-periodic points in $\bP^1(\fo_k/\fp)$ to approach zero as $|\bP^1(\fo_k/\fp)|$  tends to infinity (see \cite{FO}\cite{Bach}\cite{Silver}\cite{BGH}). In fact, in \cite{JKMT} it is shown that for any degree $d$ and $\epsilon>0$,there is a Zariski-open subset of the set of rational functions of degree $d$ such that for all $\varphi$ in this set, the proportion of periodic points is less than $\epsilon$. Moreover, in that paper specific conditions are given which ensure that a rational function belongs to this set. The conditions guarantee that the iterated extensions mentioned above are both geometric and, in some sense, as large as possible. We call the extensions \textit{geometric} when the splitting fields of $\varphi^n(x)-t$ over $k(t)$ do not contain any elements of $\bar{k}\setminus k$. However, the conditions given there are not necessary; here we extend the results by showing that we can often weaken the conditions to allow extensions that are non-geometric, but we will still require the geometric part of the extensions to be as large as possible. %This paper deals with cases like this, where the extensions are non-geometric, but we still require the geometric part of the extensions to be as large as possible.

Further, it is easy to find examples of rational functions for which the proportion of periodic points does not tend to zero. For example, the map  $\varphi_p:\bP^1(\F_p) \rightarrow \bP^1(\F_p)$ induced by $\varphi(x)=x^3+1$ is bijective whenever $\F_p$ does not contain a cube root of unity, that is, whenever $p \not\equiv 1 \mod 3$. Thus, for these primes every point is periodic under this map, which implies that the limit supremum of the proportion of periodic points is one. We will see that, with our assumptions on the geometric part of the iterated extensions, the only case when the proportion of periodic points can fail to tend to zero is when the map induces a bijection on the residue field for infinitely many primes.

Note, the rational maps that induce a bijection on $\bP^1(\fo_k/\fp)$ for infinitely many primes $\fp$ are rare in the sense that they belong to a special class of maps called \textit{exceptional maps}, which has been completely classified by Guralnick, Muller,
and Saxl \cite{GMS}. Their results generalize a conjecture of Schur from 1923 \cite{Schur}, proven by Fried in 1970 \cite{Fried}, which states that polynomials with this property are compositions
of linear polynomials and Dickson polynomials.

%A brief summary of useful results regarding the structure of the Galois groups of iterated extensions is given in Section \ref{galoisgroups}.

%\section{Galois Groups of Iterated Extensions}\label{galoisgroups}

In order to describe these results precisely we fix some notation. Let $k$ be any field and let $\varphi(x) \in k(x)$ be a rational function with degree $d>1$. Suppose also that $\varphi'(x)\neq
0$, then $\varphi^n(x)-t$ is separable for all $n$. To see this first note, it is easy to see by an induction argument that ${(\varphi^n)}'(x)\neq 0$ for all $n$. Now, let $p_n, q_n \in k[x]$ be the numerator and denominator of $\varphi^n(x)$ respectively. Note, the roots of $\varphi^n(x)-t$ are the roots of $p_n(x)-tq_n(x)$, so we want to show that $p_n(x)-tq_n(x)$ is separable. Since $p_n(x)-tq_n(x)$ is irreducible over $k(t)$, if it has
a double root then we must have $p_n'(x)-tq_n'(x)=0$ for all $x$ and hence $p_n'(x)=q_n'(x)=0$ for all $x$. This implies
${(\varphi^n)}'=\frac{q_np_n'-p_nq_n'}{(q_n)^2}=0$ for all $x$, which gives a contradiction.

Let $K_n$ be the splitting field of $\varphi^n(x)-t$ over $k(t)$ in some integral closure. Let \begin{align*}
E&=K_1\cap \overline{k}\\
A &= \Gal(E/k)\\
G&=\Gal(K_1/E(t))\\
G_n&=\Gal(K_n/E(t))\\
B_n&=\Gal(K_n/k(t)).\\
\end{align*}

We want to study the fixed point proportion of $B_n$, we can do so by studying the fixed point proportion in each coset of $B_n/G_n$. 

\begin{definition} Let $H$ be any set acting on another set $S$, we define the \textit{fixed point proportion} of $H$, denoted $\FPP(H)$, to be the proportion of elements of $H$ fixing at least one element of $S$.
\end{definition}

Note, $B_n/G_n \cong A$ for all $n$. For each $a\in A$ we let $(B_n)_a$ denote the elements of $B_n$ whose restriction to $E$ is $a$. Then if $\sigma$ is any element of $(B_n)_a$ we have $(B_n)_a=G_n\sigma$. With $\varphi$ as above,  we can consider the action of any subset of $B_n$ on the roots of $\varphi^n(x)-t$.

\begin{theorem}\label{FPP theorem}
	Let $\varphi(x)\in k(x)$ be a rational function of degree $d$. With notation as above, the following hold.
	\begin{enumerate}
		\item For any $\epsilon$ there exists $N_\epsilon$ such that if $G_{N_\epsilon}\cong[G]^{N_\epsilon}$ and each coset of $B_1/G$ has some fixed point free element, then $\FPP((B_{N_\epsilon})_a) <\epsilon$ for each $a\in A$, and hence  \[\FPP(B_{N_\epsilon})<\epsilon.\] 
		\item If $G_N\cong[G]^N$ and there is some coset $G\sigma \in B_1/G$ such that each element of $G\sigma$ has a fixed point, $\FPP((B_N)_a)=1$ where $a=\sigma|_E$, and hence \[\FPP(B_N)>\frac1{|A|}.\]
	\end{enumerate}
\end{theorem}

The hypothesis that $G_n$ is a wreath product is not too restrictive as it can be shown to hold under fairly simple conditions \cite[Lemma 3.3]{JKMT}.  In fact, if $k$ is a number field or function field, for any fixed $n$ and for ``most'' rational functions defined over $k$, $B_n=G_n=[S_d]^n$ the \textit{$n$-th iterated wreath product} of $S_d$ with itself \cite[Theorem III]{odoni}\cite[Corollary 4.3]{JO}. 

We get the following as an immediate corollary of Theorem \ref{FPP theorem}.

\begin{corollary}\label{generalFPP}
	Suppose that $G_n\cong[G]^n$ for every $n\geq 1$ and every coset of $B_1/G$ has an element acting without fixed points on the roots of $\varphi^n(x)-t$. Then \[\lim_{n \to \infty} \FPP(B_n) = 0.\] 
\end{corollary} 

Applying this result to the question of proportions of periodic points over finite fields, we are able to show the following. Note, part (a) of this theorem was shown in \cite[Proposition 6.4]{JKMT}. Let $\Norm(\fp)=\#\fo_k/\fp$ denote the norm of $\fp$. Note, $|\bP^1(\fo_k/\fp)|=\Norm(\fp)+1$.

\begin{theorem}\label{uni}
	Let $k$ be a number field, let $d >1$, and let $f(x) = x^d + c \in k[x]$ have the
	property that 0 is not preperiodic.  Then
	\begin{enumerate} 
		\item[(a)] \[ \liminf_{\Norm(\fp) \to \infty} \frac{\#\Per(f_\fp)}{\#\bP^1(\fo_k/\fp)}
		= 0;\] 
		\item[(b)] if $k$ contains a primitive $\ell$-th root of unity for $\ell | d$ with $\ell >1 $, then we have
		\[\lim_{\Norm(\fp) \to \infty} \frac{\#\Per(f_\fp)}{\#\bP^1(\fo_k/\fp)}
		= 0;\]
		\item[(c)] if $k$ does not contain a primitive $\ell$-th root of unity for any $\ell|d$ with $\ell>1$, 
		\[\limsup_{\Norm(\fp) \to \infty} \frac{\#\Per(f_\fp)}{\#\bP^1(\fo_k/\fp)}
		= 1.\]
		
	\end{enumerate}
\end{theorem}

We prove Theorem \ref{FPP theorem} in Section \ref{nongeometric}. Then in Section \ref{CS}, we prove Theorem \ref{uni} along with a more general result.

%%%%%%%%%%%%%%%%%%%%%%%%%%%%%%%%%%%%%%%%%%%%%%%%%%%%%%%

\section{Fixed Point Proportions}\label{nongeometric}

The arguments in this sections are similar to, but more general than, those of Lemmas 4.2 and 4.3 in \cite{odoni}.

\begin{definition}\label{probability argument} Let $\Gamma$ be a set of permutations acting on a finite set $S$. The  \emph{indicatrix} of $\Gamma$ is given by \[\Phi_\Gamma(x):= \frac 1{|\Gamma|} \sum_{\gamma \in \Gamma} x^{\tr \gamma}\] where $\tr \gamma := \#\{s \in S : \gamma(s)=s\}$. 
\end{definition}

\begin{lemma}\label{indicatrix} Let $a \in A=\Gal(E/k)$. If $G_n\cong [G]^n$, then \[\Phi_{(B_n)_a}(x)=\Phi_{(B_1)_a}^n(x).\]
\end{lemma}

\begin{proof} 
	By Lemma \ref{wreath subgroup}, we can embed $B_n$ in $[B_1]^n$. Thus, identifying $(B_n)_a$ with its image, we can write any $\sigma \in (B_n)_a$ as $(\pi;\tau_1,\ldots,\tau_{d^{n-1}}) \in ([B_1]^{n-1})[B_1]$ where $\pi\in B_{n-1}$, which is a subgroup of $[B_1]^{n-1}$, and $\tau_i\in B_1$. More specifically, since $\sigma$ restricts to $a$ on $E$ we must have $\pi\in (B_{n-1})_a$ and $\tau_i\in (B_1)_a$. Also, since $G_n\cong[G]^n$ we have $[K_n:K_{n-1}]=|G|^{d^{n-1}}=|(B_1)_a|^{d^{n-1}}$, so a simple degree argument shows $(B_n)_a =\{(\pi;\tau_1,\ldots,\tau_{d^{n-1}}):\pi\in(B_{n-1})_a, \tau_i\in (B_1)_a\}$. Let $S=\{\alpha_1,\ldots, \alpha_{d^{n-1}}\}$ be the set of roots of $\varphi^{n-1}(x)-t$. Then, we have 
	
	\begin{align*} 
	\Prob\big(\tr \sigma=k\big) &=\sum_l \Prob\big(\tr (\pi;\tau_1 ,\ldots,\tau_{d^{n-1}} )=k \;\vert \; \tr \pi  = l \big) \Prob\big(\tr \pi =l\big)\\
	&=\sum_l \Prob\bigg(\sum_{\substack{\alpha_i \in S \\ \pi (\alpha_i)=\alpha_i}}\tr\tau_i =k \;\vert \tr \pi = l\bigg) \Prob\big(\tr \pi =l\big)\\
	&=\sum_l \Prob\big(\tr\tau_1 +\ldots+\tr\tau_l  =k \;\vert\; \tau_j \in (B_1)_a\big) \Prob\big(\tr \pi =l\big).\\
	\end{align*}
	
	Clearly, for any set $\Gamma$, $\Phi_\Gamma(x)= \sum_k \Prob\big(\tr \gamma=k\big) x^k$, so 
	
	\begin{equation*}
	\begin{split} 
	\Phi_{(B_n)_a}(x) & = \sum_k \sum_l \Prob\big(\tr\tau_1 +\ldots+\tr\tau_l =k \;\vert\; \tau_j \in (B_1)_a\big) \Prob\big(\tr \pi =l\big) x^k \\
	&= \sum_l \Prob\big(\tr \pi =l\big) \sum_k \Prob\big(\tr\tau_1 +\ldots+\tr\tau_l  =k\;\vert\; \tau_j \in (B_1)_a\big) x^k \\
	&= \sum_l \Prob\big(\tr \pi =l\big) \bigg(\sum_m \Prob\big(\tr\tau  =m \;\vert\; \tau \in (B_1)_a\big) x^m\bigg)^l \\
	&= \sum_l \Prob\big(\tr \pi =l\big) \big(\Phi_{(B_1)_a}(x)\big)^l \\
	&= \Phi_{(B_{n-1})_a}\big(\Phi_{(B_1)_a}(x)\big).
	\end{split}
	\end{equation*}
	Thus, $\Phi_{(B_n)_a}(x)= \Phi_{(B_1)_a}^n(x)$ by induction on n.
\end{proof}

\begin{proof}[Proof of Theorem \ref{FPP theorem}] 
	
	As noted above, for any set $\Gamma$, we have \[\Phi_\Gamma(x)= \sum_k \Prob(\tr \gamma=k) x^k\] so the constant term of $\Phi_\Gamma(x)$ is the proportion of elements of $\Gamma$ with no fixed points. 
	
	First we suppose there is some $a \in A$ such that each element of $(B_1)_a$ has a fixed point. Let $\Phi(x):=\Phi_{(B_1)_a}(x)$, then $\Phi(0)=0$ so $\Phi^n(0)=0$ for all $n$. Since $G_N\cong [G]^N$ we can apply Lemma \ref{indicatrix} which implies that every $\sigma \in (B_N)_a$ must have a fixed point.  Thus, $\FPP((B_N)_a)=1$ and $\FPP(B_N) > \frac{|(B_N)_a|}{|B_N|}=\frac1{|A|}$.
	
	Now consider the case where $(B_1)_a$ has some fixed point free element for each $a \in A$. Fix an $a$ in $A$, and again let $\Phi(x):=\Phi_{(B_1)_a}(x)$. We will show $\lim_{n \rightarrow \infty} \Phi^n(0) =1$. 
	
	First, we show $\Phi'(1) =1$. Note, for any $\sigma \in (B_1)_a$ we have $(B_1)_a=G\sigma$. Let $S_1$ be the roots of $\varphi(x)-t$.
	\begin{align*}\Phi'(1) &=\frac 1{|G\sigma|}\sum_{g\in G} \tr g\sigma \\ & = \frac 1{|G|} |\{(g,\alpha)\in G\times S_1|g\sigma \alpha=\alpha\}| \\ &= \frac 1{|G|} \sum_{\alpha\in S_1} |\{g\in G|g\sigma \alpha=\alpha\}|\\
	\end{align*} 
	Note, $\{g\in G|g\sigma \alpha=\alpha\}=\Stab_G(\alpha)g_0$ where $g_0$ is any element of $\{g\in G|g\sigma \alpha=\alpha\}$. To see this, note $g \in \Stab_G(\alpha)$ if and only if $gg_0\sigma \alpha=g\alpha=\alpha$ if and only if $gg_0\in \{g\in G|g\sigma \alpha=\alpha\}$. Thus $|\{g\in G|g\sigma \alpha=\alpha\}|=|\Stab_G(\alpha)|$, and by the orbit/stabilizer theorem
	\begin{align*}
	\Phi'(1) = \frac 1{|G|} \sum_\alpha |\Stab_G(\alpha)|
	& = \frac 1{|G|} \sum_\alpha \frac{|G|}{|\Orb_G(\alpha)|} \\
	& =\sum_{\text{orbits }Y}\sum_{\alpha\in Y}\frac 1{|Y|} \\
	& =\sum_{\text{orbits }Y}1 \\
	& =1,
	\end{align*}
	since $G$ is transitive. 
	
	It is easy to see that $\Phi(1)=\frac 1{|(B_1)_a|}\sum_{\sigma\in (B_1)_a} 1^{\tr \sigma}=1$ and the first and second derivatives of $\Phi(x)$ are positive on $(0,1]$, so the graph of $\Phi(x)$ is concave up and increasing on $(0,1]$.
	
	Since $(B_1)_a$ has some fixed point free element, $\Phi(0)>0$. So the graph of $\Phi(x)$ must lie above the graph of $y=x$ on the interval $[0,1)$. Thus, $0 \leq x <\Phi(x)<1$ for $0\leq x <1$. Then $\{\Phi^n(0)\}_n$ is a strictly increasing sequence and is bounded above by $1$ so it must converge. Since $\lim_{n\rightarrow \infty} \Phi^n(0)$ must be a fixed point of $\Phi(x)$, this limit must be $1$. 
	
	Fix $\epsilon>0$. For each transitive subgroup $H$ of $S_d$ and $\sigma$ in $S_d$ such that $H\sigma$ has some element with no fixed points, the above argument implies there is some $n(H,\sigma)$ such that ${\Phi_{H\sigma}}^{n(H,\sigma)}(0)>1-\epsilon$. Let $N$ be the maximum of all $n(H,\sigma)$ taken over all $(H,\sigma)$ where $H$ is a transitive subgroup of $S_d$ and $\sigma\in S_d$ such that $H\sigma$ has an element with no fixed points. 
	
	Then for each $a$, taking any $\sigma \in (B_1)_a$ we have 
	\[\FPP((B_N)_a)=(1-\Phi_{(B_N)_a}(0))=(1-{\Phi_{(B_1)_a}}^N(0))=(1-{\Phi_{G\sigma}}^N(0))<\epsilon.\]
	This implies 
	\[\FPP(B_N)=\frac1 {|A|}\sum_{a\in A} \FPP((B_N)_a)<\epsilon.\]
\end{proof}

Corollary \ref{generalFPP} follows immediately. We can also see that the following lemma from \cite{odoni}, follows as an immediate corollary to Theorem \ref{FPP theorem}. We will use each of these in the proof our main theorem on proportions of periodic points.

\begin{corollary}[\cite{odoni}, Lemma 4.3]\label{indi}
	If $G_n\cong [G]^n$ for all $n$, then \[\lim_{n\rightarrow \infty} \FPP(G_n)=0.\]
\end{corollary}

\begin{proof}
	By Theorem \ref{FPP theorem} it suffices to show that $G$ must contain a fixed point free element. To see this, note that by Burnside's lemma \[\frac{1}{|G|}\sum_{g\in G} \tr g =1,\]
	since $G$ is transitive. So the average of $\tr g$ as $g$ varies over $G$ is $1$, and $G$ contains an element, namely the identity, with trace $d>1$. Thus, $G$ must contain an element with trace $0$.
\end{proof}

%%%%%%%%%%%%%%%%%%%%%%%%%%%%%%%%%%%%%%%%%%%%%%%%%%%%%%%

\section{Proportions of Periodic Points}\label{CS}

Let $\varphi(x)$ be a rational function of degree $d>1$ defined over a number field $k$. 
Let $p(x), q(x) \in \fo_k[x]$ such that
$\varphi(x) = p(x)/q(x)$ and let $P(X,Y)$ and $Q(X,Y)$ be the degree $d$ homogenizations of $p$ and $q$ respectively; that is, $P(X,Y) = Y^d p(X/Y)$ and $Q(X,Y) = Y^d q(X/Y)$.  

If $\fp$ is a nonzero prime in the ring of integers $\fo_k$, we say that the rational function $\varphi(x)$,
defined as above, has \textit{good reduction} at $\fp$ if we can choose $p$ and $q$ such that at least one coefficient of $p$ or $q$ has $\fp$-adic absolute value $1$, and $P(X,Y)$ and $Q(X,Y)$ have no common zeros $[\alpha,\beta]\in \bP^1(\fo_k/\fp)$.

%Define $P_1 = P$ and
%$Q_1 = Q$ and define $P_n$ and $Q_n$ by $P_n(X,Y) =
%P(P_{n-1}(X,Y), Q_{n-1}(X,Y))$ and $Q_n(X,Y) = Q(P_{n-1}(X,Y),
%Q_{n-1}(X,Y))$ for $n \geq 1$.  Then, if we set $p_k = P_k(X,1)$ and $q_k
%= Q_k(X,1)$, we have $\varphi^n(x)=\frac{p_n(x)}{q_n(x)}$. From this we can see $\varphi^n(x)$ has good reduction if $\varphi(x)$ has good reduction.

For any prime $\fp$ of $\fo_k$ of good reduction we consider the rational function $\varphi_\fp(x)$, the reduction of $\varphi$ modulo $\fp$. We examine the proportion of elements of $\bP^1(\fo_k/\fp)$ which are periodic points of $\varphi_\fp$. 

\begin{definition} Let $T:\cU\rightarrow \cU$ be any map of a set $\cU$ to itself. For $u \in \cU$ define $T^0(u)=u$ and $T^n=T(T^{n-1}(u))$. We say that $u$ is \emph{periodic} if $T^k(u)=u$ for some $k \in \mathbb N$ and we say $u$ is \emph{preperiodic} if $T^k(u)$ is periodic for some $k \in \mathbb Z_{\geq 0} $. We denote the set of periodic points $\Per(T)$ if the set $\cU$ is clear from the context. 
\end{definition}

We first note, if $\fp$ is a prime of good reduction for
$\varphi$, then the number of periodic points for $\varphi_\fp$
is bounded above by $\# \varphi_\fp^n (\bP^1(\fo_k/\fp))$ for any $n$.   %This follows from the following lemma, proven here for completeness.

\begin{lemma}[\cite{JKMT}, Lemma 5.2]\label{S}  If $\cU$ is finite then every point of $\cU$ is
	preperiodic and $\Per(T)=\cap_{n=0}^\infty T^n(\cU)$. In
	particular, $\#\Per(T) \leq \# T^n(\cU)$ for any positive
	integer $n$.  
\end{lemma}

%\begin{proof}
%Suppose $\cU$ is finite and let $u\in \cU$. By the pigeonhole principle we must have $T^m(u)=T^n(u)$ for some $m, n>0$ with $m\neq n$. Thus, $u$ is preperiodic. 
%
%Suppose  $u\in \cU$ is periodic, so we can write $T^k(u)=u$ for some $k>0$. Then for any $n\in \bN$, $T^{nk}(u)=u$ and we have $u\in T^{nk}(\cU)=T^n(T^{n(k-1)}(\cU))\subseteq T^n(\cU)$. Hence, $u\in T^n(\cU)$ for all $n$.
%	
%Now suppose $u\in\cap_{n=0}^\infty T^n(\cU)$. Then for each $i$ there is some $a_i\in\cU$ with $T^i(a_i)=u$. Since $\cU$ is finite, the pigeonhole principle implies that there exist $i,j$, with $i<j$, such that $a_i=a_j$. Then we have
%\[u=T^j(a_j)=T^{j-i}(T^i(a_j))=T^{j-i}(T^i(a_i))=T^{j-i}(u)\]
%and hence $u$ is periodic.
%\end{proof}

Let $k$, $\varphi(x)$ be as above, and let $K = k(t)$.  For $n\in\mathbb{Z}^+$, let $K_n$ be a splitting field of $\varphi^n(x)
- t$ over $K$, $B_n=\Gal(K_n/K)$, $E_n=K_n\cap \overline{k}$, and $G_n=\Gal(K_n/E_n(t))$ (note, this is slightly different from our earlier notation).

\begin{proposition}\label{for_use}
	Let $\delta > 0$.  There is a constant
	$M_\delta$ such that for all $\fp$ with $\Norm(\fp) > M_\delta$, we have
	\begin{equation}\label{proportion}
	\frac{\#\Per(\varphi_\fp)}{\#\bP^1(\fo_k/\fp)} \leq [E_n:k]\FPP(B_n) + \delta.
	\end{equation}
\end{proposition}

In order to prove Proposition \ref{for_use} we must consider how the Galois groups of the splitting fields of iterates of rational maps behave under reduction. The following two results from \cite{odoni} and \cite{JKMT} are in this direction.  Where convenient we use $\Gal(\psi(x)/K)$ to denote the Galois group of the splitting field of $\psi(x)$ over $K$ in some algebraic closure of $K$.

\begin{lemma}[\cite{odoni}, Lemma 2.4]\label{specialization} Let $A$  be an integrally closed domain with field of fractions $K$, let $K'$ be any field, and let $\psi:A\rightarrow K'$ be a ring homomorphism. Define $\widetilde{\psi} :A[x] \rightarrow K'[x]$ by $a_dx^d+a_{d-1}x^{d-1}+\ldots+a_0 \mapsto \psi(a_d)x^d+\psi(a_{d-1})x^{d-1}+\ldots+\psi(a_0)$. If $f(x)=a_dx^d+\ldots+a_0$ is a polynomial in $A[x]$ with $d\geq 1$ and $a_d \notin \ker(\psi)$, such that $\widetilde{\psi}(f(x))$ is separable over $K'$ then $f(x)$ is separable over $K$ and $\Gal(\widetilde{\psi}(f(x))/K')$ is isomorphic to a subgroup of $\Gal(f(x)/K)$. Also, if $\widetilde{\psi}(f(x))$ is irreducible over $K'$ then $f(x)$ is irreducible over $K$.
\end{lemma}

Now, let $R$ be a Noetherian integral domain of characteristic 0 and let $A$ be a finitely
generated $R$-algebra that is an integrally closed domain.  Let $h(x) \in A[x]$ be a
nonconstant polynomial that is irreducible in $F(A)[x]$. Here $F(D)$ denotes the field of fractions of an integral domain $D$. Let $B
= A[\theta_1, \dots, \theta_n]$ where $\theta_i$ are the roots of $h$
in some splitting field for $h$ over $F(A)$.  
Suppose that $F(R)$ is algebraically closed in both $F(A)$ and $F(B)$. For $\fp\in \Spec R$, let $(A)_\fp$  denote $A/\fp A
\otimes_R F(A/\fp)$ and let $h_\fp$ denote the image of $h \in (A)_\fp[x]$ under the
reduction map from $A$ to $(A)_\fp$.

\begin{lemma}[\cite{JKMT}, Proposition 4.1]\label{EGA}   There is a nonempty Zariski-dense open subset $W$ of $\Spec R$ such that for $\fp \in W$, we have $(A)_\fp$ is an integral domain, there is a bijection between the roots of $h(x)$ and the roots of $h_\fp(x)$, and the action of $\Gal (h_\fp(x) / F((A)_\fp))$ on the roots of $h_\fp(x)$ is isomorphic to the action of $\Gal(h(x) / F(A))$ on the roots of $h(x)$. 
\end{lemma}

We now fix some notation that will be used in the proof of Proposition \ref{for_use} and in the next lemma.  Let $\fp \in \Spec \fo_k$ be a prime of good reduction for $\varphi$
and let $\F_q$ denote its residue field $\fo_k / \fp$.  We let
$\varphi_\fp$ denote the reduction of $\varphi$ modulo $\fp$, 
%let $(K)_\fp$ denote $\F_q(t)$, 
let
$(K_n)_\fp$ denote the splitting field of $\varphi_\fp^n(x) - t$, and let $B_{n,\fp}=\Gal((K_n)_\fp/\F_q(t))$. 

Let $\tau$ be a degree one prime of $\F_q(t)$, that is, a degree one prime in $\F_q[t]$ or $\F_q[\frac{1}{t}]$, such that $\tau$ does not ramify in $(K_n)_\fp$. Then 
for each prime $\fm$ in $(K_n)_\fp$ lying over $\tau$, there is a unique
Frobenius element $\Frob(\fm/\tau)$ such that
$\Frob(\fm/\tau)$ fixes $\fm$ and acts as $x \mapsto x^q$ on
the residue field of $\fm$.  We let $\Frob\left(\frac{(K_n)_\fp/\F_q(t)}{\tau}\right)$
denote the conjugacy class of $\Frob(\fm/\tau)$ in $\Gal((K_n)_\fp/\F_q(t))$
(note that elements of this conjugacy class correspond to
$\Frob(\fm'/\tau)$ as $\fm'$ ranges over all primes of $(K_n)_\fp$
lying over $\tau$).  Let
$z$ be a root of $\varphi_\fp^n(x) - t$ in $(K_n)_\fp$ and let $\cS$
denote the conjugates of $z$ in $(K_n)_\fp$.

\begin{lemma}\label{Frob2}
	Let $\alpha \in \bP^1(\F_q)$ such that the degree one prime corresponding to $\alpha$ does not ramify in $(K_n)_\fp$. There exists $\beta \in \bP^1(\F_q)$ such that $\varphi_\fp^n(\beta) = \alpha$
	if and only if 
	$\Frob\left(\frac{(K_n)_\fp/\F_q(t)}{\alpha}\right)$ has a fixed point in $\cS$.
\end{lemma}

\begin{proof} 
	
	Let $A_n$ be the integral closure of $\F_q[t]$ in $(K_n)_\fp$.  Then
	$A_n^{B_{n,\fp}} = \F_q[t]$; that is, $\F_q[t]$ is the set of elements of $A_n$
	that are fixed by every element of $B_{n,\fp}$.  
	
	Now, let $(t-\alpha)$ be a
	degree one prime in $\F_q[t]$ that does not ramify in
	$(K_n)_\fp$, and let $D(\fm | (t-\alpha))$ be the decomposition group of a prime
	$\fm$ in $(K_n)_\fp$ that lies over $(t-\alpha)$.  Then, by Lemma 3.2 of
	\cite{GTZ}, the number of degree one primes in 
	$\F_q(t,z)=\F_q(z)$
	lying over $(t-\alpha)$ is equal to the number of fixed points of
	$D(\fm | (t-\alpha))$ acting on $\cS$.  That is, the number of $\beta\in \bP^1(\F_q)$ such that $\varphi^n(\beta)=\alpha$ is equal to the number of fixed points of
	$D(\fm | (t-\alpha))$ acting on $\cS$.
	
	Likewise, working with the
	integral closure $A'_n$ of $\F_q[\frac{1}{t}]$ in $(K_n)_\fp$, we see
	that if $\tau$ is the prime at infinity in $\F_q(t)$ (that is the prime
	$(\frac{1}{t})$ in $\F_q[\frac{1}{t}]$) and $\tau$ does not ramify in
	$(K_n)_\fp$, then the number of degree one primes in $\F_q(t,z)=\F_q(z)$ lying
	over $\tau$ is equal to the number of fixed points of $D(\fm |
	\tau)$ acting on $\cS$, where $\fm$ is a prime of $A'_n$ lying over
	$\tau$.  
	
	Since decomposition groups over unramified primes are
	generated by Frobenius elements, we see that for any $\alpha
	\in \bP^1(\F_q)$ that does not ramify in $(K_n)_\fp$, there exists $\beta \in \bP^1(\F_q)$ such that $\varphi_\fp^n(\beta) = \alpha$
	if and only if 
	$\Frob\left(\frac{(K_n)_\fp/\F_q(t)}{\alpha}\right)$ has a fixed point in $\cS$.
	
\end{proof}

\begin{proof}[Proof of Proposition \ref{for_use}]
	For any $\fp \in \Spec \fo_k$ of good reduction for $\varphi$ consider $\fq\in \Spec \fo_{E_n}$ lying over $\fp$. Since $\fo_k/\fp \subseteq \fo_{E_n}/\fq$ and $\fo_k/\fp \cong \F_q$, we must have $\fo_{E_n}/\fq \cong \F_{q^m}$ for some $m \geq 1$, with $m\leq[E_n:k]$. We identify $\fo_{E_n}/\fq$ with $\F_{q^m}$. Let $\widetilde{\varphi}$ be the image of $\varphi$ in $E_n(x)$, then $\widetilde{\varphi}$ has good reduction mod $\fq$. We assumed $E_n$ is algebraically closed in $K_n$, so we may apply Lemma \ref{EGA}. Since any dense subset of $\Spec \fo_{E_n}$ contains all but finitely many primes in $\Spec \fo_{E_n}$, it follows from Lemma \ref{EGA}
	that for all but finitely many $\fq\in \Spec \fo_{E_n}$, the action of $\Gal( (K_n)_\fq /
	\F_{q^m}(t))$ on the roots of $\widetilde{\varphi}^n_\fq(x) - t$ is isomorphic to the action of $G_n$ on the roots
	of $\varphi^n(x) - t$, where $(K_n)_\fq$ is the splitting field for $\widetilde{\varphi}^n(x)-t$ over $\F_{q^m}(t)$. 
	
	Note, any root of $\varphi_\fp^n(x)-t$ is also a root of $\widetilde{\varphi}^n_\fq(x)-t$ and since $E_n\subset K_n$, we have $\F_{q^m}\subset (K_n)_\fp$, so $(K_n)_\fq\cong(K_n)_\fp$  Thus, for all but finitely many $\fp \in \Spec \fo_k$,  the action of $G_{n,\fp}:=\Gal( (K_n)_\fp /
	\F_{q^m}(t))$ on the roots of $\varphi^n_\fp(x) - t$ is isomorphic to the action of $G_n$ on the roots
	of $\varphi^n(x) - t$. 
	
	Let $\pi$ denote the number of degree one primes $\alpha$ of
	$\bP^1_{\F_q}$ such that $\alpha$ does not ramify in $(K_n)_\fp$, and for any conjugacy class $C$ of $B_{n,\fp}$, let
	$\pi_C$ denote the number of degree one primes $\alpha$ of
	$\bP^1_{\F_q}$ such that $\alpha$ does not ramify in $(K_n)_\fp$ and $\Frob\left(\frac{(K_n)_\fp/\F_q(t)}{\alpha}\right) = C$.  Then \cite[Theorem 1]{murty} states that if $C$ restricts to $x\mapsto x^q$ on $\F_{q^m}$, then
	\begin{equation}\label{MS}
	\left|  \pi_C - m\pi \frac{\#C}{\#B_{n,\fp}}  \right| \leq   2q^{1/2} \left( g_{(K_n)_\fp}\frac{\#C}{\#B_{n,\fp}}  +{\#C}\right)  + (1+{\#C})\#R
	\end{equation}
	otherwise, $\pi_C=0$.
	Here $g_{(K_n)_\fp}$ is the genus of $(K_n)_\fp$ and $R$ is the set of primes
	of $\bP^1_{\F_q}$ that ramify in $(K_n)_\fp$.  Let $\Fix(B_{n,\fp})$ be the set of
	elements of $B_{n,\fp}$ that fix a root of $\varphi^n_\fp(x)-t$.  Then
	$\frac{\#\Fix(B_{n,\fp})}{\#B_{n,\fp}} = \FPP(B_{n,\fp})$, and for any $\alpha$ outside
	of $R$, there is a $\beta$ in $\bP^1(\F_q)$ such that $\varphi_\fp^n(\beta)
	= \alpha$ if and only if $\Frob\left(\frac{(K_n)_\fp/\F_q(t)}{\alpha}\right) \subseteq \Fix(B_{n,\fp})$ by Lemma
	\ref{Frob2}.  There are at most $\#R$ ramified primes $\alpha$ of $\F_q$ such
	that $\alpha \in \varphi_\fp^n(\bP^1(\F_q))$. %Note, there are $\#G_{n,\fp}$ elements of $B_{n,\fp}$ which restrict to $x\mapsto x^q$ on $\F_{q^m}$. 
	Thus,  summing the estimates
	in \eqref{MS} over all conjugacy classes in $\Fix(B_{n,\fp})$ and diving by $q+1$, we then obtain
	\begin{equation}\label{again}
	\frac{\#\varphi_\fp^n(\bP^1(\F_q))}{q+1}  \leq m\FPP(B_{n,\fp}) + \frac{2q^{1/2}}{q+1}\left( g_{(K_n)_\fp} +\#B_{n,\fp}\right)  + \frac{(c+\#B_{n,\fp}+1) \#R}{q+1},
	\end{equation}
	where $c$ is the number of conjugacy classes of in $\Fix(B_{n,\fp})$.   
	
	Note that $\deg \varphi_\fp = \deg \varphi$
	since $\varphi$ has good reduction at $\fp$. Now, since $\varphi_\fp$ ramifies over at most $(2 \deg \varphi - 2)$
	points and $\varphi_\fp^n$ can only ramify over these points and their first
	$n-1$ iterates under $\varphi_\fp$, $\varphi_\fp^n$ ramifies over at most $n (2
	\deg \varphi  - 2)$ points.   
	
	The size of $B_{n,\fp}$ can be bounded in
	terms of $n$ and $d$, since it is a subgroup of the symmetric
	group on $d^n$ elements. Also, for any $\fp$ of characteristic
	greater than $\deg \varphi$ there is no wild
	ramification at $\varphi_\fp$. Thus, $g_{(K_n)_\fp}$ can be bounded
	in terms of $\deg \varphi_\fp$ and $n$ by the Riemann-Hurwitz theorem;
	\[g_{(K_n)_\fp} \leq \#B_{n,\fp} n (2 \deg \varphi_\fp - 2).\]  Hence, by \eqref{again}
	there is an $M_\delta$ such that for all $\fp$ with $\Norm(\fp)=q \geq M_\delta$, we
	have
	\begin{equation*}
	\frac{\#\varphi_\fp^n(\bP^1(\F_q))}{q+1}  \leq m\FPP(B_{n,\fp}) + \delta.
	\end{equation*}
	Recall for all but finitely many $\fp$, $G_{n,\fp}\cong G_n$. If necessary we increase $M_\delta$ so that $G_{n,\fp}\cong G_n$, for all $\fp$ with  $\Norm(\fp)>M_\delta$.
	
	Now we show $m\FPP(B_{n,\fp}) \leq [E_n:k]\FPP(B_n)$. First note, 
	\[
	m\FPP(B_{n,\fp})=m\cdot\sum_{G_{n,\fp}\sigma\in B_{n,\fp}/G_{n,\fp}}\frac{\#\Fix(G_{n,\fp}\sigma)}{\#B_{n,\fp}}\leq\sum_{G_{n,\fp}\sigma\in B_{n,\fp}/G_{n,\fp}}\frac{\#\Fix(G_{n,\fp}\sigma)}{\#G_{n,\fp} }.\]
	By Lemma \ref{specialization}, $B_{n,\fp}$ is isomorphic to a subgroup of $B_n$. Also, we assumed $G_{n,\fp}\cong G_n$ and hence $B_{n,\fp}/G_{n,\fp}$ is isomorphic to a subgroup of $B_{n}/G_n$. Thus we have
	\[\sum_{G_{n,\fp}\sigma\in B_{n,\fp}/G_{n,\fp}}\frac{\#\Fix(G_{n,\fp}\sigma)}{\#G_{n,\fp} }\leq \sum_{G_n\sigma\in B_{n}/G_n}\frac{\#\Fix(G_n\sigma)}{\#G_n}\]
	and finally, 
	\[\sum_{G_n\sigma\in B_{n}/G_n}\frac{\#\Fix(G_n\sigma)}{\#G_n }=\frac{\#\Fix(B_n)}{\#G_n }={[E_n:k]}\FPP(B_n).\]
	
	Thus, 
	\begin{equation*}
	\frac{\#\varphi_\fp^n(\bP^1(\F_q))}{q+1}  \leq [E_n:k]\FPP(B_{n}) + \delta.
	\end{equation*}
	
	Applying Lemma \ref{S} then finishes our proof.

\end{proof}

We immediately deduce the following as a consequence Proposition \ref{for_use}.  

\begin{corollary}\label{zero}
	With notation as in Proposition \ref{for_use}, suppose that there is some $\ell$ such that $E_n=E_\ell$ for all $n\geq \ell$.  Then, if $\lim_{n \to
		\infty} \FPP(B_n) = 0$, we have 
	\begin{equation}\label{ze} \lim_{\Norm(\fp) \to \infty} \frac{\#\Per(\varphi_\fp)}{\#\bP^1(\fo_k/\fp)
		} = 0.
	\end{equation}
\end{corollary}

%(take out and cite?)
%\begin{theorem}\label{CDT}(The Chebotarev Density Theorem, see \cite{FJ}\cite{LS}\cite{murty}) Let $F$ be a global field, let $K$ be a finite Galois extension of $F$, and let $G=\Gal(K/F)$. For any conjugacy class $C$ of $G$, the Dirichlet density of the set of primes $\fp$ of $F$ for which $\Frob\left(\frac{K/F}{\fp}\right)= C$ exists and is equal to $\#C/\#G$. Furthermore, if $F$ is a number field or $F$ is a function field whose constant field is algebraically closed in $K$, then the natural density of this set also exists and is equal to $\#C/\#G$.
%\end{theorem}

We now turn back to the case $G_n\cong [G]^n$.
The following theorem provides a set of conditions which ensure that this holds. 

\begin{theorem}[\cite{JKMT}, Theorem 3.1]\label{main theorem}
	Suppose $\varphi(x) \in k(x)$ is a rational function of degree $d\geq
	2$ such that $\varphi'(x)\neq 0$. Fix $N\in\mathbb{N}$ and suppose
	there is a subset $S \subseteq \varphi_\fc$
	such that the following holds:
	\begin{enumerate}
		\item  for any $a\in S$, $b \in \varphi_{\fc}$,  and $m,n\leq N$, we have
		$\varphi^m(a)\neq \varphi^n(b)$ unless $a=b$ and $m=n$; and
		\item  the group $G$ is generated by the ramification groups of the
		$\varphi(a)$ for $a \in S$, that is 
		\[ \big \langle \bigcup_{a \in S} \; \; \bigcup_{\substack{\fp|\varphi(a)-t \\
				\text{in }K_1/E(t)}}  I(\fp|\varphi(a)-t) \big\rangle = G.  \]
	\end{enumerate}
	Then we have $G_N \cong [G]^N$.
\end{theorem}

In general, it is not too hard to show that $B_n$ and $G_n$ must each be subgroups of $[S_d]^n$ the \textit{$n$-th iterated wreath product} of $S_d$ with itself \cite[Lemma 4.1]{odoni}. The possibilities for $B_n$ and $G_n$ can be restricted further.  

\begin{lemma}[\cite{JKMT}, Lemma 3.3] \label{wreath subgroup} The group $B_n$ is isomorphic to a subgroup of $[B_1]^n$, the $n$-th iterated wreath product of $B_1$. Similarly, $G_n$ is isomorphic to a subgroup of $[G]^n$.
\end{lemma}

If the first extension $K_1/k(t)$ is not geometric, that is, if $k$ is not algebraically closed in $K_1$, then $B_n$ will not be a full iterated wreath product, as is noted in the next proposition.

\begin{proposition}[\cite{JKMT}, Proposition 3.6]\label{not-w} Let $B = \Gal(K_1/k(t))$.
	Suppose that $k$ is not algebraically closed in $K_1$, then
	$\Gal(K_n/k(t))$ is a proper subgroup of $[B]^n$ for $n>1$.
\end{proposition}

%%%%%%%%%%%%%%%%%%%%%%%%%%%%%%%%%%%%%%%%%

If $\Gal(\varphi^n(x)-t/\overline{k}(t))\cong[G]^n$ for all $n$, where $G=\Gal(\varphi(x)-t/\overline{k}(t))$ then by Proposition \ref{not-w} $E_n=E_1$ for all n, and Corollary \ref{zero} allows us to decide exactly when the proportion of periodic points of $\varphi_\fp$ tends to zero as $\Norm(\fp)$ tends to infinity. This criteria is presented in the next theorem, which extends Theorem 1.3 from \cite{JKMT}. Moreover, using Theorem 3.1 from \cite{JKMT} (Theorem \ref{main theorem} above), it is possible to find specific examples to which we can apply this criteria, the case $\varphi(x)=x^d+c$ where $0$ is not a preperiodic point of $\varphi$ is treated in Theorem \ref{uni}, which extends Proposition 6.4 \cite{JKMT}. 

\begin{theorem}\label{collide}
	Let $\varphi$ be a rational function with degree $d > 1$ such that $\Gal(\varphi^n(x)-t/\overline{k}(t))\cong[G]^n$ for all $n$, where $G=\Gal(\varphi(x)-t/\overline{k}(t))$.
	Then 
	\begin{enumerate}
		\item[(a)] \[ \liminf_{\Norm(\fp) \to \infty} \frac{\#\Per(\varphi_\fp)}{\#\bP^1(\fo_k/\fp)}
		= 0; \] 
		\item[(b)] if each coset of $\Gal(\varphi(x)-t/k(t))/\Gal(\varphi(x)-t/\overline{k}(t))$ contains at least one fixed point free element, then we have
		\[\lim_{\Norm(\fp) \to \infty} \frac{\#\Per(\varphi_\fp)}{\#\bP^1(\fo_k/\fp)}
		= 0;\]
		\item[(c)] if there is some coset of $\Gal(\varphi(x)-t/k(t))/\Gal(\varphi(x)-t/\overline{k}(t))$ for which every element of the coset fixes a root of $\varphi(x)-t$, then we have \[ \limsup_{\Norm(\fp) \to \infty} \frac{\#\Per(\varphi_\fp)}{\#\bP^1(\fo_k/\fp)}
		= 1. \] 
	\end{enumerate}
\end{theorem}

\begin{lemma}[\cite{JKMT}, Lemma 6.3]\label{eb2}
	Let $k$ be a number field, let $\varphi \in k[x]$, let $\fp$ be a
	prime of good reduction for $\varphi$, let $k'$ be a finite extension
	of $k$, and let $\tp$ denote the extension of $\varphi$ to
	$\bP^1_{k'}$.  Then
	\[ \liminf_{\substack{\Norm(\fp) \to \infty \\ \text{primes $\fp$
				of $k$}}} \frac{ \# \Per (\varphi_\fp)}{\#\bP^1(\fo_k/\fp)} \leq
	\limsup_{\substack{\Norm(\fq) \to \infty \\ \text{primes $\fq$
				of $k'$}}} \frac{ \# \Per (\tp_\fq)}{\#\mathbb{P}^1(\mathfrak{o}_{k'}/\mathfrak{q})}\]
\end{lemma}

Let $\psi:\bP^1(\F_q)\rightarrow\bP^1(\F_q)$ be a rational function, $B=\Gal(\psi(x)-t/\F_q(t))$, $G=\Gal(\psi(x)-t/\F_{q^m}(t))$, where $\F_{q^m}$ is the algebraic closure of $\F_q$ in the splitting field of $\psi(x)-t$. Then $B/G$ is isomorphic to $\Gal(\F_{q^m}/\F_q)$ and hence $B/G$ is cyclic.

\begin{lemma}\label{surjective} If every $b\in B$ with $<Gb>=B/G$ fixes at least one root of $\psi(x)-t$, then $\psi:\bP^1(\F_q)\rightarrow\bP^1(\F_q)$ is bijective.
\end{lemma}

\begin{proof} This follows directly from \cite[Lemma 4.3 and Proposition 4.4]{GTZ}.
\end{proof}

\begin{proof}[Proof of Theorem \ref{collide}]
	As before we let $K_n$ denote the
	splitting field of $\varphi^n(x) - t$ over $k(t)$. Let $E$ denote the algebraic closure of $k$ in $K_1$, then $G=\Gal( K_1 / E(t) )$. Since $\Gal(K_n/E(t))\subseteq [G]^n$ for all $n$, we have $\Gal(K_n/E(t))=\Gal(K_n/\overline{k}(t))\cong[G]^n$ for all $n$. Thus, by Corollary \ref{indi} and Corollary \ref{zero},
	\[ \lim_{\substack{ \Norm(\fq) \to \infty \\ \text{$\fq$ a prime of
				$E$} }} \frac{\#\Per( {\widetilde \varphi}_\fq)} {\Norm(\fq) + 1}
	= 0.\]
	Then (a) follows from Lemma \ref{eb2} and (b) follows from Theorem \ref{FPP theorem} and Corollary \ref{zero}.
	
	Now suppose there is some $\sigma \in \Gal(K_1/k(t))$ such that every element of the coset $G\sigma$ in $\Gal(K_1/k(t))/\Gal(K_1/E(t))$  fixes some root of $\varphi(x)-t$.  Let $a$ be the restriction of $\sigma$ to $E$. In the proof of Proposition \ref{for_use}, we saw that for all but finitely many $\fp\in \Spec \fo_k$, the action of $G_{1,\fp}=\Gal((K_1)_\fp/\F_{q^m}(t))$ on the roots of $\varphi_\fp(x)-t$ is isomorphic to the action of $G$ on the roots of $\varphi(x)-t$, where $\F_{q^m}$ is the algebraic closure of $\fo_k/\fp$ in $(K_1)_\fp$, let $U$ be the set of such primes. Also, by the Chebotarev density theorem there is a positive proportion of unramified primes in $\fo_k$ such that $\Frob\left(\frac{E/k}{\fp}\right)=C$, where $C$ is the conjugacy class of $a$ in $\Gal(E/k)$. Let $V$ denote the set of primes with this property.
	
	Let $\fp$ be any prime in $V\cap U$. Then since $\Frob\left(\frac{E/k}{\fp}\right)=C$, there is some prime $\fq$ of $E$ lying over $\fp$ such that $\Frob(\fq/\fp)=a$. Then since $\sigma|_E=a$, $\sigma$ induces an automorphism $\sigma_\fp$ of $(K_1)_\fp$ over $\fo_k/\fp \cong \F_q$ which is also an automorphism of $\fo_E/\fq$.
	% and $g$ induces and automorphism $g_\fp$ of $(K_1)_\fp$ over $\fo_E/\fq \cong \F_{q^m}$, 
	%so $\widetilde{a}_\fp\in B_{1,\fp}$.
	% and $g_\fp \in G_{1,\fp}$. 
	Let $\theta_1,\ldots,\theta_d$ be the roots of $\varphi(x)-t$. Since $\fp\in U$, there is a bijection $r_\fp$ between $\{\theta_1,\ldots,\theta_d\}$ and the roots of $\varphi_\fp(x)-t$, with
	$\sigma_\fp(r_\fp(\theta_i))=r_\fp(\sigma(\theta_i))$. Since there is some $i$ such that $\sigma(\theta_i)=\theta_i$, we have  $$\sigma_\fp(r_\fp(\theta_i))=r_\fp(\sigma(\theta_i))=r_\fp(\theta_i),$$ so $\sigma_\fp$ fixes a root of $\varphi_\fp(x)-t$. 
	
	Now, $a=\Frob(\fq/\fp)$ which implies $\sigma_\fp$ restricts to the map $x\mapsto x^q$ on $\F_{q^m}$, and $G_{1,\fp}\sigma_\fp$ is a generating coset for $B_{1,\fp}/G_{1,\fp} \cong \Gal(\F_{q^m}/\F_q)$. Then, Lemma \ref{surjective} implies that $\varphi_\fp:\bP^1(\fo_k/\fp)\rightarrow \bP^1(\fo_k/\fp)$ is surjective. Thus, we see that $\varphi_\fp^n$ is surjective for all $n$, so $\Per(\varphi_\fp)=\bP^1(\fo_k/\fp)$ by Lemma \ref{S} and $$\frac{\#\Per(\varphi_\fp)}{\#\bP^1(\fo_k/\fp)}=1.$$ Since $U\cap V$ is infinite, $\frac{\#\Per(\varphi_\fp)}{\#\bP^1(\fo_k/\fp)}=1$ for infinitely many primes and we have \[ \limsup_{\Norm(\fp) \to \infty} \frac{\#\Per(\varphi_\fp)}{\#\bP^1(\fo_k/\fp)}
	= 1. \] 
	% Let $I$ be the
	%  subgroup of $G$ generated by the ramification groups of the
	%  critical points.  Then $I=G$ by Remark \ref{all critical points}.  Thus, we have $\Gal( (\varphi^n(x) - t) / E(t) )
	%  \cong [G]^n$ for all $n$ by Theorem \ref{main theorem}. Thus, by
	%  Corollary \ref{zero},
	%\[ \lim_{\substack{ \Norm(\fq) \to \infty \\ \text{$\fq$ a prime of
	%      $E$} }} \frac{\#\Per( {\widetilde \varphi}_\fq)} {\Norm(\fq) + 1}
	%= 0,\]
	% and Lemma \ref{eb2} then implies (a).  If $k$ is
	%algebraically closed in $K_1$, then $k=E$ and (b) follows from
	%Corollary \ref{zero}.  
\end{proof}

We are now ready to prove Theorem \ref{uni}.

\begin{proof}[Proof of Theorem \ref{uni}]
	Note, $E = k(\zeta)$ where $\zeta=\zeta_d$ is $d$-th root of unity,  and
	the splitting field of ${ f(x)} - t$ over $k(t)$ is 
	$E(t)(\sqrt[d]{t - c})$ which has degree $d$ over $E(t)$. The prime $t-c$ ramifies completely in $E(t)(\sqrt[d]{t - c})$; thus,
	the Galois group is generated by the ramification group over $t-c$.
	Since the critical point 0 is not preperiodic, we see then that the
	conditions of Theorem \ref{main theorem} are met for all $N$.   Thus,
	for any $n$, we have   $\Gal( ({ f^n
		(x)} - t ) / E(t))
	\cong [C_d]^n$ where $C_d$ is the cyclic group of order $d$, and (a) follows from Theorem \ref{collide}.
	
	Suppose $k$ contains a primitive $\ell$-th root of unity for $\ell | d$, then $\zeta^{d/\ell}$ is fixed by any $\sigma\in\Gal(E/k)$. Let $\alpha=\sqrt[d]{t - c}$ be any $d$-th root of $t-c$. As noted above, $G=\Gal(f(x)-t/E(t))\cong C_d$, and the elements of $G$ are defined by $\alpha \mapsto \zeta^j\alpha$ for $0\leq j\leq d-1$. Define $\sigma_m$ by $\sigma_m(\zeta)=\zeta^m$. Then, for each $m$ such that $\sigma_m$ belongs to $\Gal(E/k)$ we get a coset $\{\tau_{m,j} : 0\leq j\leq d-1\}$ in $B_1/G$ where $\tau_{m,j}(\zeta^i\alpha)=\zeta^{mi+j}\alpha$. 
	
	The homomorphism $\tau_{m,j}$ fixes a root of $f(x)-t$ if and only if $\tau_{m,j}(\zeta^i\alpha)=\zeta^i\alpha$ for some $1\leq i\leq d$. This holds if and only if $mi+j\equiv i \mod d$ for some $1\leq i\leq d$ if and only if $d|(m-1)i+j$ for some $1\leq i\leq d$. Since $\sigma_m$ fixes $\zeta^{d/\ell}$ we must have $m\equiv 1 \mod \ell$ so if $j$ is relatively prime to $\ell$ there is no $i$ such that $d|(m-1)i+j$, which means that $\tau_{m,j}$ is fixed point free. Thus, each coset $\{\tau_{m,j} : 1\leq j\leq d\}$ has a fixed point free element and by Theorem \ref{collide} (b) holds.
	
	On the other hand, if $k$ does not contain a primitive $\ell$-th root of unity for any $\ell>1$ dividing $d$, then the map $\sigma_m$ defined by $\zeta\mapsto\zeta^m$ belongs to $\Gal(E/k)$ for any $m$ relatively prime to $d$. Note, $d$ must be odd since if $2|d$ the hypotheses for (c) cannot hold since clearly $-1\in k$. Thus, $(2,d)=1$ and the map defined by $\zeta\mapsto\zeta^2$ belongs to $\Gal(E/k)$. We will show that each element of $\{\tau_{2,j}:0\leq j\leq d-1\}$ has a fixed point. For any $j$, $\tau_{2,j}(\zeta^i\alpha)=\zeta^{2i+j}\alpha$ so $\zeta^i\alpha$ is fixed by $\tau_{2,j}$ if and only if $2i+j\equiv i \mod d$, which holds if and only if $d|i+j$, which must hold for some $i$ between $0$ and $d-1$. Thus, $\tau_{2,j}$ has a fixed point for each $j$. Hence, (c) also follows from Theorem \ref{collide}.
\end{proof}

\subsection*{Acknowledgements}
The author would like to thank Thomas J. Tucker and Rafe Jones for their useful comments and suggestions.

\newcommand{\etalchar}[1]{$^{#1}$}
\providecommand{\bysame}{\leavevmode\hbox to3em{\hrulefill}\thinspace}
\providecommand{\MR}{\relax\ifhmode\unskip\space\fi MR }
% \MRhref is called by the amsart/book/proc definition of \MR.
\providecommand{\MRhref}[2]{%
	\href{http://www.ams.org/mathscinet-getitem?mr=#1}{#2}
}
\providecommand{\href}[2]{#2}


\begin{thebibliography}{BGH{\etalchar{+}}13}
	%% Use the widest label as the parameter.
	%% Reference items can be numbered or have labels of your choice, as below.
	
	%% In IMPAN journals, only the title is italicized; boldface is not used.
	%% Our software will add links to many articles; for this, enclosing volume numbers in { } is helpful
	%% Do not give the issue number unless the issues are paginated separately.
	
	
	
	%%%%%%% To ease editing, use normal size:
	
	\normalsize
	\baselineskip=17pt
	
	%%%%%%%%%%%%%%%
	\bibitem[Bac91]{Bach}
	E.~Bach, \emph{Toward a theory of {P}ollard's rho method}, Inform. and Comput.
	\textbf{90} (1991), no.~2, 139--155.
	
	\bibitem[BGH{\etalchar{+}}13]{BGH}
	R.~L. Benedetto, D.~Ghioca, B.~Hutz, P.~Kurlberg, T.~Scanlon, and T.~J. Tucker,
	\emph{Periods of rational maps modulo primes}, Math. Ann. \textbf{355}
	(2013), no.~2, 637--660.
	
	\bibitem[FO90]{FO}
	P.~Flajolet and A.~M. Odlyzko, \emph{Random mapping statistics}, Advances in
	cryptology---{EUROCRYPT} '89 ({H}outhalen, 1989), Lecture Notes in Comput.
	Sci., vol. 434, Springer, Berlin, 1990, pp.~329--354.
	
	\bibitem[Fri70]{Fried}
	M.~Fried, \emph{On a conjecture of {S}chur}, Michigan Math. J. \textbf{17}
	(1970), 41--55.
	
	\bibitem[GMS03]{GMS}
	R.~M. Guralnick, P.~M{\"u}ller, and J.~Saxl, \emph{The rational function
		analogue of a question of {S}chur and exceptionality of permutation
		representations}, Mem. Amer. Math. Soc. \textbf{162} (2003), no.~773,
	viii+79.
	
	\bibitem[GTZ07]{GTZ}
	R.~M. Guralnick, T.~J. Tucker, and M.~E. Zieve, \emph{Exceptional covers and
		bijections on rational points}, Int. Math. Res. Not. IMRN (2007), Art. ID
	rnm004, 20.
	
	\bibitem[HJM15]{HJM}
	S. Hamblen, R. Jones, and K. Madhu, \emph{The density of primes in
		orbits of {$z\sp d+c$}}, Int. Math. Res. Not. IMRN (2015), no.~7, 1924--1958.
	
	
	
	\bibitem[Jon08]{Jones2}
	R.~Jones, \emph{The density of prime divisors in the arithmetic dynamics of
		quadratic polynomials}, J. Lond. Math. Soc. (2) \textbf{78} (2008), no.~2,
	523--544.
	
	\bibitem[Jon15]{Jones}
	\bysame, \emph{Fixed-point-free elements of iterated monodromy groups}, Trans.
	Amer. Math. Soc. \textbf{367} (2015), no.~3, 2023--2049.
	
	\bibitem[Juu16]{JO}
	J.~Juul, \emph{Iterates of generic polynomials and generic rational functions},
	Available at {\tt arXiv:1410.3814}, 25 pages, 2016.
	
	\bibitem[JKMT15]{JKMT}
	J.~Juul, P.~Kurlberg, K.~Madhu, and T.~Tucker, \emph{Wreath products and
		proportions of periodic points}, Int. Math. Res. Not. IMRN (2015), Art. ID
	rnv273, 26.
	
	\bibitem[KMS94]{murty}
	V.~Kumar~Murty and J.~Scherk, \emph{Effective versions of the {C}hebotarev
		density theorem for function fields}, C. R. Acad. Sci. Paris S\'er. I Math.
	\textbf{319} (1994), no.~6, 523--528.
	
	
	\bibitem[Odo85]{odoni}
	R.~W.~K. Odoni, \emph{The {G}alois theory of iterates and composites of
		polynomials}, Proc. London Math. Soc. (3) \textbf{51} (1985), no.~3,
	385--414.
	
	\bibitem[Sch23]{Schur}
	I.~Schur, \emph{{\"U}ber den zusammenhang zwischen einem problem der
		zahlentheorie und einem satz \" uber algebraische funktione n}, S.-B. Preuss.
	Akad. Wiss., Phys.â€“Math. Klasse (1923), 123--134.
	
	\bibitem[Sil08]{Silver}
	J.~H. Silverman, \emph{Variation of periods modulo {$p$} in arithmetic
		dynamics}, New York J. Math. \textbf{14} (2008), 601--616.
	


\end{thebibliography}
\end{document}